\documentclass[12pt,centertags]{amsart}
\usepackage{amsmath,amstext,amsthm,a4,amssymb,amscd}
\usepackage[mathscr]{eucal}
\usepackage{mathrsfs}
\usepackage{epsf}
\numberwithin{equation}{section}

\newtheorem{thm}{Theorem}[section]
\newtheorem{lemma}[thm]{Lemma}
\newtheorem{prop}[thm]{Proposition}
\newtheorem{cor}[thm]{Corollary}
\theoremstyle{definition}
\newtheorem{rem}[thm]{Remark}
\newtheorem{nota}[thm]{Notation}
\theoremstyle{definition}
\newtheorem{defn}[thm]{Definition}
\newcommand{\be}{\begin{eqnarray}}
\newcommand{\ee}{\end{eqnarray}}

\newcommand{\comment}[1]{}

\begin{document}

\title{A mod 2 index theorem for pin$^{-}$ manifolds}
\author{Weiping Zhang}\thanks{Reset from the original MSRI Preprint No. 053-94.}

\address{Chern Institute of Mathematics \& LPMC, Nankai
University, Tianjin 300071, P.R. China}
\email{weiping@nankai.edu.cn}

\begin{abstract}    We establish a mod 2 index theorem for real vector bundles over $8k+2$ dimensional compact pin$^{-}$ manifolds. The analytic index is the reduced $\eta$ invariant of (twisted) Dirac operators and the topological index is defined through $KO$-theory. Our main result extends the mod 2 index theorem of Atiyan and Singer  to non-orientable manifolds. 
\end{abstract}

\maketitle
\tableofcontents

\setcounter{section}{-1}

\section{Introduction} \label{s0}

Let $B$ be an $8k+2$ dimensional compact pin$^-$ manifold. By this we always assume a pin$^-$ structure has been chosen on $TB$. Let $E$ be a real vector bundle over $B$. By introducing suitable metrics and connections on $TB$ and $E$, one can define a self-adjoint `twisted' Dirac operator $\widetilde{D}_{B,E}$ on $B$ with coefficient $E$. The reduced $\eta$ invariant \cite{APS} of $\widetilde{D}_{B,E}$, denoted by $\overline{\eta}( \widetilde{D}_{B,E})$, turns out to be mod 2 independent of the metrics and connections appeared in the definition of $\widetilde{D}_{B,E}$. Thus $\overline{\eta}( \widetilde{D}_{B,E})$ is a mod 2 topological invariant (in fact a pin$^-$ cobordism invariant, as we will see in the main text). 

The purpose of this paper is to give a purely topological formula for this analytically defined invariant. Our motivation of proving such a formula comes from the Rokhlin type congruence formulas we proved in Zhang \cite{Z1}, where pin$^-$ manifolds appear as obstructions to the existence of spin structures on oriented manifolds.

Now suppose $B$ is orientable and carries an orientation. Then $B$ is a spin manifold carrying a spin structure induced from the pin$^-$ structure. The reduced $\eta$ invariant turns out to be the mod 2 analytic index defined by Atiyah and Singer \cite{ASV}. 

Recall that in this case a topological index was defined by Atiyah-Singer \cite{ASV} and an equality between the analytic and topological indices was established in \cite{ASV}. 

Our topological interpretation of the reduced $\eta$ invariants for pin$^-$ manifolds is inspired by Atiyah-Singer's construction. The topological index we will define will turn out to lie in ${\bf Z}[\frac{1}{2}]$ (mod 2).

After making clear what should be proved, we find the result follows from an easy modification of the paper by Bismut-Zhang \cite{BZ} where a Riemann-Roch property for reduced $\eta$ invariants on odd dimensional manifolds was formulated and proved. In fact, a direct proof of the Atiyah-Singer mod 2 index theorem \cite{ASV} along the lines of \cite{BZ} has already been worked out in Zhang \cite{Z2}. One thing to be remarked is that while in \cite{Z2}, one need not use the local index techniques in \cite{BZ}, here for non-orientable manifolds, the full strength of the techniques in \cite{BZ}, which in turn rely on Bismut-Lebeau \cite{BL}, should be used. 

Twisted Dirac operators and their reduced $\eta$ invariants were first studied by Gilkey \cite{G} for pin$^c$ manifolds. In \cite{St}, Stolz studied the reduced $\eta$ invariants on pin$^+$ $(8k+4)$-manifolds and used them to detect for example the exotic ${\bf R}P^4$ constructed by Cappell and Shaneson \cite{CS}. The same method here can be used to give unified topological formulas for these reduced $\eta$ invariants too. The modifications are fairly easy and will not be carried out in this paper. 

Our results suggest that one can use the reduced $\eta$ invariants to detect pin$^-$ cobordism classes. 

Also our definition of the topological index seems to be closely related to the $KR$-theory developed by Atiyah \cite{At} and hopefully will find applications in real algebraic geometry. In fact one of the first applications of the original Rokhlin congruence \cite{R} lies in real algebraic geometry. 

This paper is organized as follows. In the first section, we recall some algebraic preliminaries which will be used in the rest of this paper. In Section 2 we define twisted Dirac operators and the associated analytic index. Section 3 contains the definition of the topological index. In Section 4 we establish an equality between the analytic and topological indices defined in Sections 2 and 3 respectively, based on a Riemann-Roch property for the analytic index. This Riemann-Roch property will be proved in Section 5. There is also an Appendix in which we prove an extended Rokhlin congruence formula not included in Zhang \cite{Z1}. The mod 2 indices studied in the main text appear most naturally in this version of Rokhlin congruences.

\section{Algebraic preliminaries}\label{s1}

In this Section, we recall some elementary algebraic facts for the completeness of this paper. A standard reference is the paper of Atiyah, Bott and Shapiro \cite{ABS}. One can also consult Lawson-Michelsohn's book \cite{LM}. 

This section is organized as follows. In a), we recall the basic definitions of pin$^-$ group and their representations. We pay special attention to dimensions $8k+2$ and $8k+3$ which are essential for this paper. In b), we recall the real structure of the spinor representations in dimension $8k$. This plays the basic role in our definition of the topological index in Section 3. In c), we recall a factorization formula for pin$^-$ representations. 

$\ $

{\it a).  Pin$^-$ groups and their representations}

Let $E$ be an $n$ dimensional oriented Euclidean space. Let $c(E)$ be the real Clifford algebra of $E$. That is, $c(E)$ is spanned over ${\bf R}$ by $1,\, e$, $e\in E$ and the commutation relations $ee'+e'e=-2\langle e,e'\rangle$. The pin$^-$ group in dimension $n$, pin$^-(n)$, is the multiplication group generated by $e\in E\subset c(E)$, $\|e\|=1$. Let $\chi$ be the representation $\chi: {\rm pin}^-(n)\rightarrow O(1)$ given by
\begin{align} \label{1.1}
\chi:\ e_{i_1}\cdots e_{i_j}\mapsto (-1)^{j},\ \ i_l\neq i_k,\ \ l\neq k.
\end{align}

Let $\gamma: {\rm pin}^-(n)\rightarrow O(n)$ be the canonical representation defined by $\gamma(e)(\omega)=e\omega e$ for $\omega,\,e\in E\subset c(E)$, $\|e\|=1$. Let $\Delta$ be a pin$^-$ module. Then one verifies that
\begin{align} \label{1.2}
\omega(ew)=\omega e\omega^{-1}(\omega)w=\chi(\omega)(\gamma(\omega)e)(\omega w)
\end{align}
for $\omega\in{\rm pin}^-(n)$, $e\in E$ and $w\in\Delta$.

Thus, the Clifford action 
\begin{align} \label{1.3}
c: \, E\otimes\Delta\rightarrow \chi\otimes\Delta
\end{align}
is pin$^-$ equivariant 

We now assume $n=8k+3$. Then by Atiyah-Bott-Shapiro \cite{ABS}, $c(E)={\rm End}_{\bf H}(S_+)\oplus {\rm End}_{\bf H}(S_-)$, $\dim S_+=\dim S_-=2^{4k}$.  Let $e_1,\,\cdots,\,e_n$ be an oriented orthonormal basis of $E$. Set $s_n=e_1\cdots e_n$. Then $S_\pm$ are characterized by $s_n$, acting on $S_\pm$ as $\pm{\rm Id}$. We will fix $S_+$ as {\it the} irreducible module of $c(E)$ as well as pin$^-(n)$. One has $S_-=\chi\otimes S_+$ and the Clifford action $c:E\otimes S_+\rightarrow S_-$ is pin$^-(n)$ equivariant. Also $S_\pm$ carry naturally induced metrics. 

Now let $G$ be a Euclidean space of dimension $8k+2$. Let $E={\bf R}\oplus G$. Then viewing $G\subset E$, we can view $S_+(E)$ as a pin$^-(G)$ module and have the Clifford action 
\begin{align} \label{1.4}
c: \, G\otimes S_+\rightarrow S_-.
\end{align}
Let $e\in G^\perp\subset E$, $\|e\|=1$. One then composes (\ref{1.4}) to a pin$^-$ equivariant action
\begin{align} \label{1.5}
c(e)c: \, G\otimes S_+\rightarrow S_+.
\end{align}

Without any confusion, we will note $S_+(G)$ from now on for $S_+$. This is sometimes called a tangential representation. It plays a fundamental role in the construction of twisted Dirac operators in Section 2. 

$\ $

{\it b). Spin representations in dimension $8k$}

Now let $E$ be an $8k$ dimensional oriented Euclidean space. Then by \cite{ABS}, 
\begin{align} \label{1.6}
c(E)={\rm End}_{\bf R}(F)=F\otimes F^*,
\end{align}
where $F=F_+\oplus F_-$ is the ${\bf Z}_2$-graded Euclidean space of $E$-spinors. Let $e_1,\,\cdots,\,e_{8k}$ be an oriented orthonormal basis of $E$. Then $F_\pm$ are characterized by $s_{8k}=e_1\cdots e_{8k}$ acting on $F_\pm$ as $\pm{\rm Id}$. 

The real space structure of   the Clifford algebra of an $8k$ dimensional space and the corresponding irreducible spin representations play important roles in the definition of the topological index in Section 3. 

Now let $E^*$ be the dual of $E$ carrying the dual metric. If $e\in E$, let $e^*\in E^*$ correspond to $e$ by the scalar product of $E$. If $e\in E$, let $c(E)$, $\widehat c(e)$ be the operators acting on $\Lambda(E^*)$, 
\begin{align} \label{1.7}
c(e)= e^*\wedge -i_e,\ \ \ \ \ \ \ \widehat c(e)=e^*\wedge+i_e.
\end{align}

Recall that we have the identification of ${\bf Z}$-graded real vector spaces
\begin{align} \label{1.8}
c(E)=\Lambda\left(E^*\right).
\end{align}

Let $\sigma$ be $1$ on $c^{\rm even}(E)$, $-1$ on $c^{\rm odd}(E)$. Then $\sigma$ also acts in the obvious way on $\Lambda(E^*)$. Under the identification (1.8), $c(e)$ is exactly the left Clifford multiplication by $e$ and $\widehat c(e)\sigma$ is the right Clifford action by $e$. Also, as $F$ is a $c(E)$ module, (1.6) is an identification of the left and right Clifford modules. 

Let $\tau=\pm 1$ on $F_\pm$, $\tau^*=\pm 1$ on $F^*_\pm$. Then $\tau,\ \tau^*$ act on $F\otimes F^*$ as $\tau\otimes 1$, $1\otimes\tau^*$. One easily verifies that
\begin{align} \label{1.9}
\tau^*=\sigma\tau.
\end{align}

$\ $

{\it c). Decomposition of pin$^-$ equivariant representations}

Now let $G$ be an $8k+2$ dimensional oriented Euclidean space and $E$ an $8l$ dimensional oriented Euclidean space. Then by using the notation as in the previous subsections, one gets
\begin{align} \label{1.10}
c(G\oplus E) =c(G)\widehat\otimes c(E),
\end{align}
where $\widehat\otimes$ is the standard notation for super tensor product. By (1.10) one gets easily
the following pin$^-$  equivariant factorization of representations,
\begin{align} \label{1.11}
S_+(G\oplus E)=S_+(G)\otimes F(E). 
\end{align}

\section{Twisted Dirac operators on pin$^-$ manifolds and $\eta$ invariants
}\label{s2} 

In this Section, we define the twisted Dirac operators on $8k+2$ dimensional pin$^-$ manifolds with coefficients in real vector bundles. The word ``twisted" reflets the fact that these operators are defined as ``tangential operators" of Dirac operators on associated $8k+3$ dimensional manifolds. Then we recall the definition of the reduced $\eta$ invariants of twisted Dirac operators and show that these invariants are mod 2 topological invariants. 

This Section is organized as follows. In a), we recall the definition of pin$^-$ manifolds. In b), we define the twisted Dirac operators associated to vector bundles over pin$^-$ manifolds. In c), we recall the definition of reduced $\eta$ invariants and its basic properties.

$\ $

{\it a). Pin$^-$ structures on vector bundles and manifolds}

A basic reference for this subsection is Kirby-Taylor \cite{KT}.

Let $B$ be a compact manifold. Let $E$ be a rank $n$ real vector bundle over $B$.

\begin{defn}\label{t2.1}
If there is a pin$^-(n)$ principal bundle $P$ over $B$ and a representation $\rho:{\rm pin}^-(n)\rightarrow {\rm End}({\bf R}^n)$ such that $E=P\times_\rho {\bf R}^n$, then we call $E$ a pin$^-$ vector bundle over $B$ carrying a pin$^-$ structure determined by $(P,\rho)$.
\end{defn}

The following characterization of a vector bundle to be a pin$^-$ bundle can be found for example in Stolz \cite{St}.

\begin{prop}\label{t2.2}
A real vector bundle $E$ over a compact manifold $B$ is a pin$^-$ vector bundle if and only if it's Stiefel-Whitney classes satisfy the condition that $w_1^2(E)+w_2(E)=0$.
\end{prop}

\begin{defn}\label{t2.3}
If  the tangent bundle $TB$ of a compact manifold $B$ carries a pin$^-$ structure, then $B$ is called a pin$^-$ manifold. And in general by this we will mean $B$ has been equipped with a pin$^-$ structure.
\end{defn}

{\it b). The twisted Dirac operators on an $8k+2$ dimensional pin$^-$ manifold}

Let $B$ be a compact $8k+2$ dimensional pin$^-$ manifold. Then $M= (-1,0]\times B$ is an $8k+3$ dimensional pin$^-$ manifold with boundary $\partial M=\{0\}\times B$, and carries an induced pin$^-$ structure. Note $\pi:(-1,0]\times B\rightarrow B$ the projection map. 

Let $g^{TB}$ be a metric on $TB$. Let $(-1,0]$ carry the standard metric $dt^2$. Let $g^{TM}=dt^2\oplus g^{TB}$ be the product metric on $M$.

Recall that a specific  pin$^-(8k+3)$ representation has been described  in Section 1a).
Let $S_+(M)$ be the associated pinor bundle over $M$ for $(M,g^{TM})$. Then $S_+(M)$ carries a canonical metric $g^{S_+(M)}$ induced from $g^{TM}$. Also the Levi-Civita connection $\nabla^{TM}$ of $g^{TM}$ lifts to a Euclidean connection $\nabla^{S_+(M)}$ on $S_+(M)$.

Let $E$ be a real vector bundle over $B$. Let $g^E$ be a metric on $E$. Let $\nabla^E$ be a Euclidean connection on $E$. Then $\pi^*\nabla^E$ is a Euclidean connection on $(\pi^*E,\pi^*g^E)$.

Let $\nabla^{S_+(M)\otimes \pi^*E}$ be the connection on $S_+(M)\otimes\pi^*E$ defined by
\begin{align} \label{2.1}
\nabla^{S_+(M)\otimes \pi^*E}(u\otimes v)=\nabla^{S_+(M)}u\otimes v+u\otimes \pi^*\nabla^Ev
\end{align}
for $u\in\Gamma(S_+(M)),\ v\in\Gamma(\pi^*E)$. Then  $\nabla^{S_+(M)\otimes \pi^*E}$ is a Euclidean connection on $(S_+(M)\otimes \pi^*E,g^{S_+(M)}\otimes \pi^*g^E)$.

Let $o(TM)$ be the orientation bundle of $TM$.

Let $e_1,\,\cdots,\,e_{8k+2}$ be an orthonormal basis of $TB$. Then $\frac{\partial}{\partial t},\,\pi^*e_1,\,\cdots,\,\pi^*e_{8k+2}$ is an orthonormal basis of $TM$.

\begin{defn}\label{t2.4}
The Dirac operator $D_{M,\pi^*E}$ is the differential operator from $\Gamma(S_+(M)\otimes\pi^*E)$ to $\Gamma(o(TM)\otimes S_+(M)\otimes\pi^*E)$ defined by
\begin{align} \label{2.2}
D_{M,\pi^*E}=c\left( \frac{\partial}{\partial t}\right)\nabla^{S_+(M)\otimes \pi^*E}_{\frac{\partial}{\partial t}}
+\sum_{i=1}^{8k+2}c\left(\pi^*e_i\right)\nabla^{S_+(M)\otimes \pi^*E}_{\pi^*e_i}.
\end{align}
\end{defn}

Clearly, $D_{M,\pi^*E}$ is a first order elliptic differential operator.

Now since $M$ is of product structure, one has $\nabla^{S_+(M)\otimes \pi^*E}_{\frac{\partial}{\partial t}}=\frac{\partial}{\partial t}$ and we can write (\ref{2.2}) as
\begin{multline}\label{2.3}
D_{M,\pi^*E}=-c\left( \frac{\partial}{\partial t}\right)
\left( - {\frac{\partial}{\partial t}}
+ c\left( \frac{\partial}{\partial t}\right)\sum_{i=1}^{8k+2}c\left(\pi^*e_i\right)\nabla^{S_+(M)\otimes \pi^*E}_{\pi^*e_i}\right)
\\
=-c\left( \frac{\partial}{\partial t}\right)
\left( - {\frac{\partial}{\partial t}}
+\pi^*\widetilde D_{B,E}\right),
\end{multline}
where $\widetilde D_{B,E}$ is the uniquely determined differential operator on $\Gamma(S_+(M)|_B\otimes E)$. 

\begin{nota}\label{t2.5}
From now on, we will denote by $S_+(B)$ the bundle $S_+(M)|_B$.
\end{nota}

\begin{defn}\label{t2.6}
The operator $\widetilde D_{B,E}$ is called the twisted Dirac operator on $B$ with coefficient $E$. 
\end{defn}

One verifies easily that $\widetilde D_{B,E} : \Gamma(S_+(B)\otimes E) \rightarrow  \Gamma(S_+(B)\otimes E) $ is a formally self-adjoint first order differential operator. 

$\ $ 

{\it c). Reduced $\eta$ invariants as analytic indices}

The $\eta$ and reduced $\eta$ invariants were introduced by Atiyah, Patodi and Singer \cite{APS} in their study of index theorems for manifolds with boundary. We recall the definition in our context.

The $\eta$ function of $\widetilde D_{B,E} $ is defined by
\begin{align} \label{2.4}
 \eta\left( \widetilde D_{B,E} ,s\right)=\sum\frac{{\rm sgn}(\lambda)}{|\lambda|^s},\ \ {\rm Re}(s)>>0,
\end{align}
where $\lambda$ runs over the non-zero eigenvalues of $\widetilde D_{B,E} $. 

Standard methods show that $\eta(\widetilde D_{B,E} ,s)$ is a holomorphic function on the half plane ${\rm Re}(s)>>0$, and can be extended to a meromorphic function on the whole complex plane, which is holomorphic at $s=0$ \cite{APS}. 

The value of $\eta(\widetilde D_{B,E} ,s)$ at $s=0$ is called the $\eta$ invariant of $\widetilde D_{B,E} $ and is denoted by $\eta(\widetilde D_{B,E} )$. The reduced $\eta$ invariant of $\widetilde D_{B,E} $ is defined by \cite{APS}
\begin{align} \label{2.5}
 \overline \eta\left( \widetilde D_{B,E} \right)= \frac{\dim\ker \widetilde D_{B,E}  + \eta \left( \widetilde D_{B,E} \right)}{2}.
\end{align}

Now assume $B$ bounds an $8k+3$ dimensional compact pin$^-$ manifold $K$, and assume $E$ extends to a real vector bundle $\widehat E$ over $K$. Let $g^{TK}$, $g^{\widehat E}$ be the metrics on $TK$, $\widehat E$ respectively such that they restrict to $g^{TB}$, $g^E$ on the boundary and are of product structures near the boundary. Then the constructions in the previous subsection can be applied here. 

Let $\nabla^{\widehat E}$ be a Euclidean connection on $\widehat E$ which is of product structure near the boundary.

Let $D_{K,\widehat E}$ be the Dirac operator associated to $(K,g^{TK},\widehat E,g^{\widehat E}, S_+(K))$, which is defined similarly as in Definition \ref{t2.4}. We impose the Atiyah-Patodi-Singer boundary condition \cite{APS} on $D_{K,\widehat E}$. Then by the index theorem for manifolds with boundary of Atiyah-Patodi-Singer \cite{APS}, one has
\begin{align} \label{2.6}
 {\rm ind} D_{K,\widehat E}=- \overline \eta\left( \widetilde D_{B,E} \right).
\end{align}
The local index term disappears in (\ref{2.6}) simply because $\dim K$ is odd. 

Now by Section 1a), we know that $S_+(K)$ carries a quaternionic structure. Furthermore, $D_{K,\widehat E}$, as well as it's Atiyah-Patodi-Singer boundary condition, are  easily seen to be ${\bf H}$ linear. From (\ref{2.6}), one gets the following important result. 

\begin{prop}\label{t2.7}
If $B$ bounds a compact pin$^-$ manifold $K$ and $E$ extends to $K$, then $\overline\eta(\widetilde D_{B,E})$ is an even integer. 
\end{prop}

By Proposition \ref{t2.7}, we know that for any $8k+2$ dimensional compact pin$^-$ manifold $B$ and a real vector bundle $E$ over $B$, the quantity $\overline\eta(\widetilde D_{B,E})$ (mod 2) is a well defined pin$^-$ cobordism invariant. In particular, it does not depend on the metrics and connections used to define it. 

\begin{defn}\label{t2.8}
The quantity $\overline\eta(\widetilde D_{B,E})$ (mod 2) is called the mod 2 analytic index of $E$ and is denoted by ${\rm ind}^a(E)$.
\end{defn}

It is easy to verify that  ${\rm ind}^a(E)$ provides a homomorphism ${\rm ind}^a:{KO}(B)\rightarrow {\bf R}/2{\bf Z}$. The purpose of this paper is to give a unified topological formula for this analytically defined homomorphism.

\section{The topological index for vector bundles over pin$^-$ manifolds
}\label{s3} 

The purpose of this Section is to define what we call the topological index of a real vector bundle over an $8k+2$ dimensional pin$^-$ manifold.

Our construction is inspired by a construction of Atiyah and Singer \cite{ASV} (cf. also \cite{LM}), but is different in several aspects. The reason is that we here should take more care of the situation where our manifold is no longer orientable and that the group $\widetilde{KO}({\bf R}P^l)$ satisfies no periodicity of Bott type.

This Section is organized as follows. In a), we give a geometric construction of the direct image in $KO$-theory for embeddings between manifolds. In b), we discuss the pin$^-$ structures on real projective spaces. In c), we present our definition of the mod 2 topological index.

$\ $

{\it a). A geometric construction of the direct image for embeddings}

Let $i:Y\hookrightarrow X$ be an embedding of compact manifolds without boundary. Let $\pi:N\rightarrow Y$ be the normal bundle to $Y$ in $X$.

We make the assumption that ${\rm rk}(N)\equiv 0$ (mod 8) and that $N$ is an oriented spin vector bundle over $Y$, carrying a fixed spin structure. 

Let $E$ be a real vector bundle over $Y$. Then by the standard construction of Atiyah and Hirzebruch \cite{AH} and of Atiyah, Bott and Shapiro \cite{ABS} (cf. also Lawson-Michelsohn \cite{LM}), the direct image of $E$ under $i$ is a well defined element $i_!E\in\widetilde{KO}(X)$. 

In what follows we will give a concrete geometric realization of $i_!E$. This construction is inspired by Quillen's superconnection \cite{Q}. A similar construction for complex vector bundles has already appeared in Bismut-Zhang \cite{BZ}. 

Let $g^{TX}$ be a metric on $TX$. Let $g^{TY}$ be the metric on $TY$ determined by $g^{TX}$. Let $g^N$ be the induced metric on $N$ such that we have the orthogonal decomposition of vector bundles and metrics,
\begin{align}\label{3.1}
TX|_Y=TY\oplus N,
\end{align}
\begin{align}\label{3.2}
\left. g^{TX}\right|_Y=g^{TY}\oplus g^N.
\end{align}

Set for any $r>0$,
\begin{align}\label{3.3}
D_r(N)=\{n\in N|\, \|n\|\leq r\},\ \ \ \ \ \  S_r(N)=\partial D_r(N). 
\end{align}

Let $F(N)=F_+(N)\oplus F_-(N)$ be the bundle of $N$-spinors. Since ${\rm rk}(N)\equiv 0$ (mod 8), $F_\pm$ are real vector bundles by Section 1b). Let $F^*(N)=F_+^*(N)\oplus F_-^*(N)$ be the dual of $F(N)$. We are mainly interested in the ${\bf Z}_2$-graded vector bundle $F^*(N)\otimes E=F_+^*(N)\otimes E\oplus F_-^*(N)\otimes E$.

Let $n\in N$. Let $\widetilde c(n)$ be the transpose of the Clifford action $c(n)$ on $F(N)$. Then $\widetilde c(n)$ extends to an action on $F^*(N)\otimes E$ as $\widetilde c(n)\otimes{\rm Id}_E$, which we still note by $\widetilde c(n)$. 

Let now $G$ be a real vector bundle over $Y$ such that $F_-^*(N)\otimes E\oplus G$ is a trivial vector bundle over $Y$. The existence of $G$ is clear.

Consider the pair of vector bundles $\pi^*(F_\pm^*(N)\otimes E\oplus G)$ over $D_r(N)$. For any $n\in D_r(N)$, $n\neq 0$, the map
\begin{align}\label{3.4}
\widetilde c({n})\oplus{\rm Id}_{\pi^*G}:\pi^*\left(F_+^*(N)\otimes E\oplus G\right)
\rightarrow
\pi^*\left(F_-^*(N)\otimes E\oplus G\right)
\end{align}
is a linear isomorphism. In other words, $\pi^*(F_-^*(N)\otimes E\oplus G)$ and $\pi^*(F_+^*(N)\otimes E\oplus G)$
are equivalent over $D_r(N)\setminus Y$ for any $r>0$. In particular, over $S_1(N)$, we have two trivial vector bundles $\pi^*(F_\pm^*(N)\otimes E\oplus G)$ and an identifying map given by 
$v(n)=\widetilde c({n})\oplus{\rm Id}_{\pi^*G}$. This map $v$ can be trivially extended to the whole $X\setminus D_1(N)$ as an invertible map between two trivial vector bundles extending $\pi^*(F_\pm^*(N)\otimes E\oplus G)$ respectively. 

Denote by $\xi_\pm$ the two resulting vector bundles over $X$. Then it is clear that $\xi_+-\xi_-$ is a representative of the direct image $i_!E\in\widetilde{KO}(X)$ constructed in \cite{AH} and \cite{ABS}. This is our geometric realization of the direct image. We emphasize  that we have also constructed a map $v:\xi_+\rightarrow \xi_-$ between $\xi_+$ and $\xi_-$ such that $v$ is invertible on $X\setminus Y$ and that when restricted to $D_1(N)$, as $\xi_\pm=\pi^*(F_\pm^*(N)\otimes E\oplus G)$, $v$ takes the form
\begin{align}\label{3.6}
 v(n)= \widetilde c({n})\oplus{\rm Id}_{\pi^*G}
\end{align}
for $n\in D_1(N)$.

In the rest of this paper, whenever we refer to a direct image for an embedding, we will mean a geometric realization in the sense of this subsection. 

$\ $

{\it b). The pin$^-$ structures on real projective spaces}

Let $q$ be a positive integer. Let ${\bf R}P^q$ be the $q$ dimensional real projective space. Recall that the total Stiefel-Whitney class of ${\bf R}P^q$ is given by (cf. Milnor-Stasheff \cite{MS})
\begin{align}\label{3.7}
  w\left({\bf R}P^q\right)=(1+a)^{q+1},
\end{align}
where $a$ is the generator of $H^1({\bf R}P^q,{\bf Z}_2)$. 
 
From (\ref{3.7}) one deduces that
\begin{align}\label{3.8}
 w_1^2\left({\bf R}P^q\right)+w_2\left({\bf R}P^q\right)=\frac{(q+1)(3q+2)}{2}\,a^2.
\end{align}

By (\ref{3.8}), $w_1^2 ({\bf R}P^q )+w_2 ({\bf R}P^q )=0$ occurs when i), $q$ is odd and $q\equiv 3$ (mod 4) or ii), $q$ is even and $q\equiv 2$ (mod 4).

In the first case, we get a spin manifold while in the second case, we get a non-orientable pin$^-$ manifold. We will concentrate on the non-orientable case. 

On ${\bf R}P^{4k+2}$, there are exactly two different pin$^-$ structures. We will fix one as follows.

The orientation cover $S^{4k+2}$ of ${\bf R}P^{4k+2}$ bounds a disk $D^{4k+3}$ with its antipodal involution. We can extend this involution on $P=D^{4k+3}\times{\rm pin}^-(4k+3)$ by the left multiplication by $s_{4k+3}=e_1\cdots e_{4k+3}$ on the second factor, where $e_1,\,\cdots,\,e_{4k+3}$ is an oriented orthonormal basis of ${\bf R}^{4k+3}$. There is an obvious isomorphism between the associated bundle $P\times_\gamma{\bf R}^{4k+3}$ and the tangent bundle of $D^{4k+3}$. This ${\bf Z}/2$ equivariant pin$^-$ structure induces a pin$^-$ structure on ${\bf R}P^{4k+2}$. In the rest of this paper, we will always assume that ${\bf R}P^{4k+2}$ has this pin$^-$ structure. 

$\ $

{\it c). The definition of the topological index}

Let $B$ be an $8k+2$ dimensional compact pin$^-$ manifold with a fixed pin$^-$ structure. Let ${\bf R}P^{8k+2}$ be the real projective space with the pin$^-$ structure specified in the last subsection. Let $\gamma$ be the canonical real line bundle over ${\bf R}P^{8k+2}$. That is, $\gamma=o(T{\bf R}P^{8k+2})$, the orientation bundle of $T{\bf R}P^{8k+2}$. 

By a classical result of Steenrod-Whitney \cite{S}, there is a classifying map $f:B\rightarrow {\bf R}P^{8k+2}$, uniquely determined up to homotopy, such that
\begin{align}\label{3.9}
f^*(\gamma)=o(TB) ,
\end{align}
where $o(TB)$ is the orientation bundle of $TB$. 

On the other hand, one can find easily a sufficiently large integer $m$ so that there is an embedding $g:B\hookrightarrow S^{8m}$.

The maps $f$ and $g$ together give an embedding
\begin{align}\label{3.10}
h=(f,g):B\hookrightarrow {\bf R}P^{8k+2}\times S^{8m},\ \ \ \ x\mapsto (f(x),g(x)).
\end{align}

Now as $S^{8m}$ is an oriented spin manifold carrying the unique spin structure, $ {\bf R}P^{8k+2}\times S^{8m}$ is a pin$^-$ manifold carrying an induced pin$^-$ structure.

Let $N$ be the normal bundle to $B$ in ${\bf R}P^{8k+2}\times S^{8m}$. From (\ref{3.9}),  one verifies that
\begin{align}\label{3.11}
 w_1(N)=h^*\left(w_1\left({\bf R}P^{8k+2}\times S^{8m}\right)\right)-w_1(B)=f^*w_1\left({\bf R}P^{8k+2}\right)-w_1(B)=0
\end{align}
and that
\begin{multline}\label{3.12}
 w_2(N)=h^*\left(w_2\left({\bf R}P^{8k+2}\times S^{8m}\right)\right)-w_2(B)
\\
=f^*w_1^2\left({\bf R}P^{8k+2}\right)+w_2(B)=w_1^2(B)+w_2(B)=0.
\end{multline}

From (\ref{3.11}) and (\ref{3.12}), $N$ is a spin vector bundle over $B$. The orientation of $TS^{8m}$ induces an orientation on $N$. Furthermore,  the pin$^-$ structures on $B$ and ${\bf R}P^{8k+2}\times S^{8m}$ determine a spin structure on $N$.

Thus the constructions in Section 3a) apply here.

Let $E$ be a real vector bundle over $B$, we then get an element $i_!E\in \widetilde{KO}({\bf R}P^{8k+2}\times S^{8m})$. 

Now let $i_{8k+2}:{\bf R}P^{8k+2}\hookrightarrow {\bf R}P^{8k+2}$ be the identity map. Let $p$ be a point in $S^{8m}$ and let $i_p:{\bf R}P^{8k+2}\rightarrow S^{8m}$ be the collapsing map with image $p$. Then $i_{8k+2,m}=(i_{8k+2},i_p): {\bf R}P^{8k+2}\hookrightarrow {\bf R}P^{8k+2}\times S^{8m}$ is an embedding. Furthermore, $i_{8k+2,m}$ determines a map
\begin{align}\label{3.13}
\left( {i_{8k+2,m}}\right)_!: KO\left({\bf R}P^{8k+2}\right)\rightarrow \widetilde{KO}\left({\bf R}P^{8k+2}\times S^{8m}\right). 
\end{align}

By the Bott periodicity theorem (cf. \cite{LM}), $(i_{8k+2,m})_!$ is an isomorphism not depending on the choice of the point $p$.

\begin{nota}\label{t3.1}
If $a,\ b,\ c,\,\cdots\in {\bf R}$ is a series of real numbers, then we denote by ${\bf Z}\{a,\,b,\,c,\,\cdots\}$ the abelian subgroup of ${\bf R}$ generated by the numbers $a,\ b,\ c,\,\cdots $.
\end{nota}

The following result on the structure of $\widetilde{KO}({\bf R}P^{8k+2})$ is essential to our definition of the topological index. It replaces the Bott periodicity theorem in the spin case.

\begin{lemma}\label{t3.2} \mbox{{\rm (Adams \cite[Theorem 7.4]{Ad})}}. 
The group $\widetilde{KO}({\bf R}P^{8k+2})$ is an abelian group of order $2^{4k+2}$ generated by $1-\gamma$.
\end{lemma}

By this Lemma, any element $\alpha$ of $KO({\bf R}P^{8k+2})$ can be represented as
\begin{align}\label{3.14}
\alpha=m_\alpha+n_\alpha(1-\gamma),\  \ m_\alpha, \ n_\alpha\in{\bf Z},\ \ \ 0\leq n_\alpha\leq 2^{4k+2}-1.
\end{align}

Let $q_{8k+2}:KO({\bf R}P^{8k+2})\rightarrow {\bf Z}\{\frac{1}{2^{4k+2}}\}/2{\bf Z}$ be the homomorphism defined by
\begin{align}\label{3.15}
q_{8k+2}(\alpha)=\frac{m_\alpha}{2^{4k+2}} +\frac{n_\alpha}{2^{4k+1}}\ \ \ ({\rm mod}\ 2).
\end{align}

We can now give our definition of the topological index.

\begin{defn}\label{t3.3}
Let $E$ be a real vector bundle over an $8k+2$ dimensional compact pin$^-$ manifold $B$. The topological index of $E$, denoted by ${\rm ind}^t(E)$, is an element in ${\bf Z}\{\frac{1}{2^{4k+2}}\}/2{\bf Z}$ given by $q_{8k+2}(((i_{8k+2,m})_!)^{-1}h_!(E))$. 
\end{defn}

\begin{rem}\label{t3.4}
Using the standard method in $KO$-theory and the fact that the induced vector bundles are independent of the homotopy type of the induced maps, one sees easily that ${\rm ind}^t(E)$ does not depend on the interger $m$, the classifying map $f$ and the embedding $g$ appeared in the process of it's definition. Thus ${\rm ind}^t(E)$ is a well defined object.
\end{rem}

\begin{rem}\label{t3.5}
Similarly, one can also define the topological index for real vector bundles over an $8k+6$ dimensional compact pin$^-$ manifolds. In this case the value of the index will lie in ${\bf Z}\{\frac{1}{2^{4k+4}}\}/{\bf Z}$. 
\end{rem}

\section{A mod 2 index theorem for $8k+2$ dimensional pin$^-$ manifolds
}\label{s4} 

In this Section, we establish a mod 2 index theorem for real vector bundles over $8k+2$ dimensional pin$^-$ manifolds, which is the main concern of this paper. The proof relies on a Riemann-Roch formula for the analytic index which we state in Theorem \ref{t4.1}. The proof of this Riemann-Roch property will be carried out in Section 5.

This Section is organized as follows. In a), we state the Riemann-Roch formula to be proved in Section 5. In b), we will state and prove the mod 2 index theorem.

$\ $

{\it a). A Riemann-Roch formula for the analytic index}

Let $X,\ Y$ be two compact pin$^-$ manifolds of dimension $8m+2$, $8n+2$ respectively, such that there is an embedding $i:Y\hookrightarrow X$.

We make the assumption that
\begin{align}\label{4.1}
i^*w_1(TX)=w_1(TY).
\end{align}

Let $N$ be the normal bundle to $Y$ in $X$. By (\ref{4.1}), $N$ is an orientable spin vector bundle of dimension $8(m-n)$. We fix an orientation on $N$. Then $N$ carries a spin structure canonically induced from the given pin$^-$ structures on $TX$ and $TY$.

We can then apply the direct image construction in Section 3a).

The following Riemann-Roch type formula is essential.

\begin{thm}\label{t4.1}
Let $E$ be a real vector bundle over $Y$. Then the following identity holds,
\begin{align}\label{4.2}
{\rm ind}^a(E)={\rm ind}^a\left(i_!E\right).
\end{align}
\end{thm}

\begin{proof}
The proof of (\ref{4.2}) will be given in Section 5. 
\end{proof}

$\ $

{\it b). An equality between the analytic and topological indices}

We now use the notation as in Section 3a).

Again, $B$ is a compact $8k+2$ dimensional pin$^-$ manifold and $E$ is a real vector bundle over $B$.

Recall that we have constructed various direct images in the course of the definition of the topological index.

By using Theorem \ref{t4.1} two times, we have the following equality between analytic indices,
\begin{align}\label{4.3}
{\rm ind}^a(E)={\rm ind}^a\left(\left(\left(i_{8k+2,m}\right)_!\right)^{-1}h_!E\right). 
\end{align}

So in order to prove an equality between the analytic and topological indices, we need only to check it on the real projective spaces. 

\begin{lemma}\label{t4.2}
Let $\alpha$ be a real vector bundle over the pin$^-$ manifold ${\bf R}P^{8k+2}$. Then the following identity holds,
\begin{align}\label{4.4}
\overline\eta\left(\widetilde D_{{\bf R}P^{8k+2},\alpha}\right)\equiv q_{8k+2}(\alpha)\ \ \ ({\rm mod}\ 2).
\end{align}
\end{lemma}
\begin{proof}
By Lemma \ref{t3.2}, we need only to check (\ref{4.4}) for $\alpha=1$ and $\alpha=\gamma$. This can be done in exactly the same way as in Gilkey \cite{G} and Stolz \cite{St}. All what one need to do is the trivial modification of \cite[Corollary 5.4]{St} by replacing `pin$^+$' there by `pin$^-$', and by replacing `${\bf R}P^{8k+4}$' there by `${\bf R}P^{8k+2}$'.

We leave the details to the interested reader.
\end{proof}

We can now state the main result of this paper.

\begin{thm}\label{t4.3}
The following identity holds for a real vector bundle $E$ over a compact $8k+2$ dimensional pin$^-$ manifold,
\begin{align}\label{4.5}
{\rm ind}^a(E)={\rm ind}^t(E).
\end{align}
\end{thm}

\begin{proof}
By (\ref{4.3}) and Lemma \ref{t4.2}, we need finally to check that for any real vector bundle $\alpha$ over ${\bf R}P^{8k+2}$, we have
\begin{align}\label{4.6}
q_{8k+2}(\alpha)={\rm ind}^t(\alpha).
\end{align}

In fact, if we take $B$ in Section 3a) to be the ${\bf R}P^{8k+2}$ and the classifying map to be the identity, then by deforming the embedding $g:{\bf R}P^{8k+2}\hookrightarrow S^{8m}$ to the constant map $g_p$, $h=(f,g)$ deformes to an embedding $h_p=(f,g_p):{\bf R}P^{8k+2}\hookrightarrow {\bf R}P^{8k+2}\times S^{8m}$. In particular, the deformed maps in this process from $h$ to $h_p$ remain to be embeddings. Thus, we have
\begin{align}\label{4.7}
h_!(\alpha)=\left(h_p\right)_!(\alpha)=\left(i_{8k+2,m}\right)_!(\alpha).
\end{align}

The proof of Theorem \ref{t4.3} is completed. 
\end{proof}

The following consequence is of independent interest.

\begin{cor}\label{t4.5}
For any real vector bundle $E$ over a compact $8k+2$ dimensional pin$^-$ manifold $B$, if $\widetilde D_{B,E}$ is a twisted Dirac operator defined by using suitable metrics and connections on $B$, $E$ respectively, then one has $\overline\eta(\widetilde D_{B,E})\in{\bf Z}\{\frac{1}{2^{4k+2}}\}$.
\end{cor}

This extends a result of Gilkey \cite{G}.

\begin{rem}\label{t4.6}
There is an analogous mod ${\bf Z} $ index theorem for real vector bundles over $8k+6$ dimensional pin$^-$ manifolds. Details are easy to carry out and are left to the interested reader.
\end{rem}

\section{A Riemann-Roch formula for pin$^-$ manifolds
}\label{s5} 

The purpose of this Section is to prove Theorem \ref{t4.1}. Recall that a similar result for $\eta$ invariants on odd dimensional spin manifolds has already been proved in Bismut-Zhang \cite{BZ}. All what we need to do is to modify the arguments in \cite{BZ} to fit with our specific situation here.

We recall that the techniques in \cite{BZ} depends heavily on the paper of Bismut-Lebeau \cite{BL}.

This Section is organized as follows. In a), we restate Theorem \ref{t4.1} in terms of reduced $\eta$ invariants. In b), we employ some simplifying assumptions on certain metrics and connections. In c), we state six technical results. The Riemann-Roch property is proved in d), based on the intermediary results in c). These intermediary results are then proved in e).

$\ $

{\it a). A Riemann-Roch formula for reduced $\eta$ invariants}

Let $i:Y\hookrightarrow X$ be an embedding of a pair of compact pin$^-$ manifolds of dimensions $8m+2$ and $8n+2$ respectively.

As in Section 4a), we make the assumption that
\begin{align}\label{5.1}
i^*w_1(TX)=w_1(TY).
\end{align}

Let $\pi:N\rightarrow Y$  be the normal bundle to $Y$ in $X$. From (\ref{5.1}), one sees easily that $N$ is an orientable  spin bundle over $Y$ (compair with (\ref{3.11}) and (\ref{3.12})). 
We fix an orientation on $N$. Then the pin$^-$ structures on $TX$ and $TY$ determine a spin structure on $N$. And we can apply the direct image construction of Section 3a) to real vector bundles over $Y$.

We will use the notation of Section 3a).

Let $E$ be a real vector bundle over $Y$.

Recall that in Section 3a), starting with a metric  $g^{TX}$ on $TX$, we constructed two real vector bundles $\xi_\pm$ of a same dimension on $X$ such that $\xi_+-\xi_-=i_!E\in\widetilde {KO}(X)$. 

Let $g^{TY}$ be the restriction of $g^{TX}$ on $TY$.

We introduce metrics and Euclidean connections on $E$ and $\xi_\pm$ respectively.

The main result of this Section can be stated as follows for the twisted Dirac operators constructed in Section 2.

\begin{thm}\label{t5.1}
The following identity holds,
\begin{align}\label{5.2}
\overline\eta\left(\widetilde D_{X,\xi_+}\right)-\overline\eta\left(\widetilde D_{X,\xi_-}\right)
\equiv \overline\eta\left(\widetilde D_{Y,E}\right)\ \ ({\rm mod}\ 2).
\end{align}
\end{thm}

$\ $

{\it b). Some geometric simplifying assumptions}

Recall  that by Proposition \ref{t2.7}, the reduced $\eta$ invariants in (\ref{5.2}) do not depend on the metrics and connections used to define the twisted Dirac operators. So in order to prove Theorem \ref{t5.1}, we can and we will make these metrics and connections as simple as possible.

First of all, we assume that the embedding $i:(Y,g^{TY})\hookrightarrow (X,g^{TX})$ is totally geodesic. 

Let $g^N$ be the metric on $N$ so that we have the orthogonal decomposition of vector bundles and metrics
\begin{align}\label{5.3}
TX|_Y=TY\oplus N,\ \ \ \ \left.g^{TX}\right|_Y=g^{TY}\oplus g^N.
\end{align}
Then $g^N$ lifts to metrics $g^{F_\pm(N)}$, $g^{F_\pm^*(N)}$ on $F_\pm(N)$, $F_\pm^*(N)$ respectively. 

Let $P^{TY}$, $P^N$ be the orthogonal projection maps from $TX$ to $TY$ and $N$ respectively.

Let $\nabla^{TX}$, $\nabla^{TY}$ be the Levi-Civita connections associated to $g^{TX}$, $g^{TY}$ respectively. Then
\begin{align}\label{5.4}
\nabla^{TY}=P^{TY}\nabla^{TX|_Y}P^{TY}.
\end{align}

Let $\nabla^N$ be the connection defined by 
\begin{align}\label{5.5}
\nabla^N=P^N\nabla^{TX|_Y}P^N.
\end{align}
Then $\nabla^N$ is a Euclidean connection on $N$. It lifts to Euclidean connections $\nabla^{F_\pm(N)}$ and $\nabla^{F_\pm^*(N)}$ on $F_\pm(N)$ and $F_\pm^*(N)$ accordingly. 

Let $g^G$ be a metric on the vector bundle $G$ appeared in the construction of $\xi_\pm$ in Section 3a). Let $\nabla^G$ be a Euclidean connection on $G$. Then  $\xi_\pm|_Y=F_\pm^*(N)\otimes E\oplus G$ carry the metrics $g^{F_\pm^*(N)\otimes g^E}\oplus g^G$ such that $F_\pm^*(N)\otimes E$ and $G$ are orthogonal to each other, and the corresponding connections $\nabla^{F_\pm^*(N)\otimes E}\oplus\nabla^G$. We can and we do lift these metrics and connections to $\pi^*(F_\pm^*(N)\otimes E\oplus G)|_{D_r(N)}$ and then extend them to metrics $g^{\xi_\pm}$ and Euclidean connections $\nabla^{\xi_\pm}$ on $\xi_\pm$. 

Let $\xi=\xi_+\oplus\xi_-$ be the ${\bf Z}_2$-graded vector bundle. Let $g^\xi=g^{\xi_+}\oplus g^{\xi_-}$ be the metric on $\xi$ so that $\xi_+$ and $\xi_-$ are orthogonal to each other. Let $\nabla^\xi=\nabla^{\xi_+}\oplus\nabla^{\xi_-}$ be the corresponding Euclidean connection on $\xi$.

Then on $D_1(N)$, one has
\begin{align}\label{5.6}
\xi=\pi^*\left( F_+^*(N)\otimes E\oplus G\right) \oplus \pi^*\left(F_-^*(N)\otimes E\oplus G\right),
\end{align}
$$g^\xi=\pi^*\left(g^{F_+^*(N)}\otimes g^E\oplus g^G\right) \oplus \pi^*\left(g^{F_-^*(N)}\otimes g^E\oplus g^G\right) ,$$
$$\nabla^\xi=\pi^*\left(\nabla^{F_+^*(N)\otimes E}\oplus \nabla^G\right) \oplus \pi^*\left(\nabla^{F_-^*(N)\otimes E} \oplus \nabla^G\right) .$$
Furthermore, by Section 3a), there is a map $v:\xi_+\rightarrow \xi_-$, which is invertible on $X\setminus Y$, and, when restricted to $D_1(N)$, takes the form
\begin{align}\label{5.7}
 v(n)=\widetilde c(n) \oplus{\rm Id}_{\pi^*G},\ \ \ n\in D_1(N),
\end{align}
where $\widetilde{c}(n)$ is the Clifford action of $n$ on $\pi^*(F_+^*(N)\otimes E)$. The Clifford action also acts on $\pi^*(F_-^*(N)\times E).$

Let $\tau^N$ be the action on $F(N)$ such that $\tau^N|_{F_\pm(N)}=\pm{\rm Id}_{F_\pm(N)}$. 
Let $\tau^{N*}$ be the transpose of $\tau^N$. Then $\tau^{N*}$ extends to an action on $F^*(N)\otimes E$ as $\tau^{N*}\otimes {\rm Id}_E$ which we still note by $\tau^{N*}$. 

Let $v^*$ be the adjoint of $v$ with respect to $g^\xi$. Set $V=v+v^*$. Then $V$ is a self-adjoint element in ${\rm End}^{\rm odd}(\xi)$, and one has on $D_1(N)$ that
\begin{align}\label{5.8}
V(n)=\tau^{N*}\widetilde c(n)+{\rm Id}_{\pi^*G}.
\end{align}

\begin{rem}\label{t5.2}
All the simplifying conditions in Bismut-Zhang \cite[Section 1c) and 2e)]{BZ} for direct images of complex vector bundles have now real analogues. Compare also with Zhang \cite[(1.1)]{Z2}.
\end{rem}

$\ $

{\it c). Six intermediary results}

We use the same assumptions and notation as in the previous two subsections. In particular, we assume that the simplifying conditions made in the last subsection hold. 

Let $\widetilde D_{X,\xi}$ be the twisted Dirac operator defined by
\begin{align}\label{5.9}
 \widetilde D_{X,\xi}=\widetilde D_{X,\xi_+}-\widetilde D_{X,\xi_-}.
\end{align}
Then
\begin{align}\label{5.10}
\overline\eta\left(  \widetilde D_{X,\xi}\right)=\overline\eta\left(\widetilde D_{X,\xi_+}\right)-\overline\eta\left(\widetilde D_{X,\xi_-}\right)+\dim \ker \widetilde D_{X,\xi_-}.
\end{align}

Now since $S_+(X)$ has a quaternionic structure, $\dim\ker\widetilde D_{X,\xi_-}$ is an even integer. From (\ref{5.10}), one gets
\begin{align}\label{5.11}
\overline\eta\left(  \widetilde D_{X,\xi}\right)\equiv \overline\eta\left(\widetilde D_{X,\xi_+}\right)-\overline\eta\left(\widetilde D_{X,\xi_-}\right) \ \ ({\rm mod}\ 2).
\end{align}

Let $\widetilde V$ be the operator acting on $\Gamma(S_+(X)\otimes\xi)$ defined by
\begin{align}\label{5.12}
\widetilde V:\alpha\otimes \beta\mapsto\alpha\otimes V\beta
\end{align}
for $\alpha\in\Gamma(S_+(X))$, $\beta\in\Gamma(\xi)$. 

Then $\widetilde V$ is a pin$^-$ equivariant self-adjoint element in ${\rm End}^{\rm odd}(S_+(X)\otimes \xi)$.

For any $T\geq 0$, set
\begin{align}\label{5.13}
\widetilde D_{X,\xi,T}=\widetilde D_{X,\xi}+T\widetilde V.
\end{align}

For $a>0$, $T\geq 0$, let $K_T^a$ be the direct sum of the eigenspaces of the operator $\widetilde D_{X,\xi,T}$ which are associated to eigenvalues whose absolute value is strictly smaller than $a$. Let $P_T^a$ be the orthogonal projection operator from $\Gamma(S_+(X)\otimes\xi)$ on $K_T^a$. Set $P_T^{a,\perp}=1-P_T^a$.

The following results, which are similar to those of Bismut-Zhang \cite[Theorems 3.7-3.12]{BZ} for odd dimensional manifolds and complex bundles case, will play essential roles in the proof of Theorem \ref{t5.1} in the next subsection. The proof of these results will be given in Section 5e).

\begin{thm}\label{t5.3}
For any $\alpha_0>0$, there exists $C>0$ such that for $\alpha\geq\alpha_0$, $T\geq 1$, then
\begin{align}\label{5.14}
\left|{\rm Tr}\left[\widetilde D_{X,\xi,T}\exp\left(-\alpha\left(\widetilde D_{X,\xi,T}\right)^2\right)\right]
-{\rm Tr}\left[\widetilde D_{Y,E}\exp\left(-\alpha\left(\widetilde D_{Y,E}\right)^2\right)\right]\right|\leq \frac{C}{T^{1/2}}.
\end{align}
\end{thm}

\begin{thm}\label{t5.4}
For any $a>0$, there exist $c>0$, $C>0$ such that for $\alpha\geq 1$, $T\geq 1$, then
\begin{align}\label{5.15}
\left|{\rm Tr}\left[P_T^{a,\perp}\widetilde D_{X,\xi,T}\exp\left(-\alpha\left(\widetilde D_{X,\xi,T}\right)^2\right)\right] \right|\leq c\exp(-C\alpha).
\end{align}
\end{thm}

Take now $a_0>0$ such that the operator $\widetilde D_{Y,E}$ has no nonzero eigenvalues in the interval $[-2a_0,2a_0]$.

\begin{thm}\label{t5.5}
For $T $ large enough, then
\begin{align}\label{5.16}
\dim K_T^{a_0}=\dim\ker\left(\widetilde D_{Y,E}\right).
\end{align}
Moreover,
\begin{align}\label{5.17}
\lim_{T\rightarrow +\infty}{\rm Tr}\left[\left|\widetilde D_{X,\xi,T}\right|P_T^{a_0}\right]=0.
\end{align}
\end{thm}

\begin{thm}\label{t5.6}
There exist $c>0$, $\gamma\in(0,1]$ such that for $u\in(0,1]$, $0\leq T\leq 1/u$, then
\begin{align}\label{5.18}
\left|{\rm sup}(T,1){\rm Tr}\left[\widetilde V\exp\left(- \left(u\widetilde D_{X,\xi,T} +T\widetilde V\right)^2\right)\right] \right|\leq c(u(1+T))^\gamma. 
\end{align}
\end{thm}

\begin{thm}\label{t5.7}
For any $T>0$, the following identity holds,
\begin{align}\label{5.19}
\lim_{u\rightarrow 0^+} {\rm Tr}\left[\frac{T}{u}\widetilde V  \exp\left(- \left(u\widetilde D_{X,\xi,T} +\frac{T}{u}\widetilde V\right)^2\right)\right]  =0.
\end{align}
\end{thm}

\begin{thm}\label{t5.8}
There exist $c>0$, $\delta\in(0,1]$ such that for $u\in(0,1] $, $T\geq 1$, then
\begin{align}\label{5.20}
\left|{\rm Tr}\left[\frac{T}{u}\widetilde V\exp\left(- \left(u\widetilde D_{X,\xi,T} +\frac{T}{u}\widetilde V\right)^2\right)\right] \right|\leq\frac{ c}{T^\delta} .
\end{align}
\end{thm}

$\ $

{\it d). Proof of Theorem \ref{t5.1}}

We construct a closed one form on ${\bf R}_+^*\times {\bf R}_+$ and then use it to prove Theorem \ref{t5.1}, in exactly the same way as in Bismut-Zhang \cite{BZ}.

\begin{thm}\label{t5.9}
Let  $u>0$, $T\geq 0$. Let $\beta_{u,T}$ be the one form on ${\bf R}_+^*\times {\bf R}_+$,
\begin{align}\label{5.21}
\beta_{u,T}=du{\rm Tr}\left[\widetilde D_{X,\xi,T}\exp\left(-u^2\widetilde D^2_{T,\xi,T}\right)\right]+dT{\rm Tr}\left[u\widetilde V\exp\left(-u^2\widetilde D^2_{X,\xi,T}\right)\right].
\end{align}
Then the $ 1$-form $\beta_{u,T}$ is closed.

\begin{proof} Theorem \ref{t5.9} can be proved in exactly the same way as in \cite[Theorem 3.4]{BZ}.
\end{proof}
\end{thm}

\noindent {\it Proof of Theorem \ref{t5.1}}. Fix constants $\epsilon$, $A$, $T_0$ such that $0<\epsilon<1\leq A<+\infty$, $0\leq T_0<+\infty$.  Let $\Gamma_{\epsilon,A,T_0}$ be the oriented contour in ${\bf R}_+^*\times {\bf R}_+$ as constructed in \cite{BZ}, consisting of four oriented pieces,
$$\Gamma_1:\ \ T=T_0;\ \ \epsilon\leq u\leq A,$$
$$\Gamma_2:\ \ 0\leq T\leq T_0;\ \ u=A,$$
$$\Gamma_3:\ \ T=0;\ \ \epsilon\leq u\leq A,$$
$$\Gamma_4:\ \ 0\leq T\leq T_0;\ \ u=\epsilon,$$
with the counterclockwise orientation. 

For $1\leq k\leq 4$, set
$$I_k=\int_{\Gamma_k}\beta_{u,T}.$$

Then by Theorem \ref{t5.9}, one gets the identity
\begin{align}\label{5.22}
\sum_{k=1}^4I_k=0. 
\end{align}

Theorem \ref{t5.1} then follows by making in (\ref{5.22}), $A\rightarrow +\infty$, $T_0\rightarrow +\infty$ and $\epsilon\rightarrow 0$ in this order, and proceeding in exactly the same strategy as in Bismut-Zhang \cite[Sections 3e)-g)]{BZ}. All one need to  notify are the following two points. 

i). We use the intermediary results Theorems \ref{t5.3}-\ref{t5.8} here, instead of those in \cite[Section 3d)]{BZ};

ii). Since $S_+(X)$ and $S_+(Y)$ are  ${\bf H}$ linear spaces and the twisted Dirac operators $\widetilde D_{X,\xi_\pm,T}$, $\widetilde D_{Y,E}$ as well as the map $\widetilde V$ are ${\bf H}$ linear, all the mod ${\bf Z}$ terms in \cite[Sections 3e)-g)]{BZ} can and will be replaced by mod $2{\bf Z}$.

By noting these two points and by proceeding in exactly the same way as in \cite[Sections 3e)-g)]{BZ}, one gets Theorem \ref{t5.1}.

$\ $

{\it e). Proof of Theorems \ref{t5.3}-\ref{t5.8}}

The methods of Bismut-Zhang \cite[Section 4]{BZ}, which go back to Bismut-Lebeau \cite{BL}, can be adapted here with little change to prove Theorems \ref{t5.3}-\ref{t5.8}. All one need to take care are the following two points:

i). We should modify the harmonic oscillator construction in \cite[Section 4a)]{BZ} for complex spinor spaces in order to fit with the real situation here;

ii). Since we are now in the even dimensional situation, the local index techniques in \cite{BZ} should be modified. But the even dimensional case turns out to be much simpler here, and does not cause any extra difficulty than \cite{BZ}. Details are fairly easy to fill and are left to the interested reader.

So we now concentrate on the modification of the harmonic oscillator construction, which is also easy. It is included here only for completeness.

We use the notation of Section 1b).

Set $m={\rm rk}(E)\equiv 0$ (mod 8). Let $e_1,\,\cdots,\,e_m$ be an orthonormal basis of $E$ and $e_1^*,\,\cdots,\,e_m^*$ be the dual basis of $E^*$. 

Let $\Gamma(\Lambda(E^*))$ be the vector space of smooth sections of $\Lambda(E^*)$ over $E$. Let $D^E$ be the operator acting on $\Gamma(\Lambda(E^*))$ defined by
\begin{align}\label{5.23}
D^E=\sum_{i=1}^m\left(c(e_i)\otimes \tau^*\right)\nabla_{e_i}.
\end{align}

Let $Z$ be the generic point of $E$. Then $\tau^*\widetilde c(Z)$ acts on $\Gamma(\Lambda(E^*))$.

Set
\begin{align}\label{5.24}
S=\sum_{i=1}^mc(e_i)\widetilde c(e_i). 
\end{align}

\begin{prop}\label{t5.10}
The following identity holds,
\begin{align}\label{5.25}
\left(D^E+\tau^*\widetilde c(Z)\right)^2=-\Delta+|Z|^2+S.
\end{align}
\end{prop}
\begin{proof} The proof of (\ref{5.25}) is trivial.
\end{proof}

We now give another expression for $S$. Let $N$ be the number operator of $\Lambda(E^*)$, i.e., $N$ acts on $\Lambda^p(E^*)$ by multiplication by $p$.

\begin{prop}\label{t5.11}
The following identity holds,
\begin{align}\label{5.26}
S=(2N-m)\sigma.
\end{align}
\end{prop}
\begin{proof}
By the construction in Section 1b), we know that
\begin{align}\label{5.27}
\widetilde c(e_i)=\widehat c(e_i)\sigma.
\end{align}
Also one verifies easily that
\begin{align}\label{5.28}
\sum_{i=1}^mc(e_i)\widehat c(e_i)=2N-m.
\end{align}

Formula (\ref{5.26}) follows from (\ref{5.24}), (\ref{5.27}) and (\ref{5.28}). 
\end{proof}

\begin{prop}\label{t5.12}
The lowest eigenvalue of the operator $S$ is $-m$. The corresponding eigenspace is one dimensional and is spanned by $1$.
\end{prop}

\begin{proof}
For any $p$ and $\alpha\in\Lambda^p(E^*)$, one has
\begin{align}\label{5.29}
(2N-m)\sigma\alpha=(-1)^p(2p-m)\alpha.
\end{align}

Now for any $0\leq p\leq m$, one has $(-1)^p(2p-m)\geq -m$ with equality holds for $p=0$ (as $m$ is even).

Proposition \ref{t5.12} follows immediately. 
\end{proof}

From (\ref{5.25}), the operator $(D^E+\tau^*\widetilde c(Z))^2$ is of harmonic oscillator type. Therefore it has discrete spectrum and compact resolvent. 

We now have the following analogue of  \cite[Theorem 4.5]{BZ}.

\begin{thm}\label{t5.13}
The kernel of $(D^E+\tau^*\widetilde c(Z))^2$ is one dimensional and is spanned by
\begin{align}\label{5.30}
\beta=\exp\left(-\frac{|Z|^2}{2}\right).
\end{align}
Also,
\begin{align}\label{5.31}
\left(D^E+\tau^*\widetilde c(Z)\right)\beta=0.
\end{align}
\end{thm}
\begin{proof}
The kernel of the operator $-\Delta+|Z|^2-m$ is spanned by $\exp(-\frac{|Z|^2}{2})$. The first part of the theorem follows from Propositions \ref{t5.10} and \ref{t5.12}.

Equation (\ref{5.31}) is a consequence of the fact that $\beta\in\ker (D^E+\tau^*\widetilde c(Z))^2$. One can also check it directly. 
\end{proof}

We now come back to the proof of Theorems \ref{t5.3}-\ref{t5.8}. Note that near the embedded manifold $Y$, we have the following pin$^-$ equivariant factorizations via (\ref{1.10}),
\begin{align}\label{5.32}
c(TX)|_Y=c(TY)\widehat\otimes c(N),
\end{align}
and 
\begin{align}\label{5.33}
S_+(X)=S_+(Y)\otimes F(N).
\end{align}

The proof of Theorems \ref{t5.3}-\ref{t5.8} can then be proceeded with little change as in \cite[Sections 4b)-e)]{BZ}.

As we have remarked, the difference in the local index calculation causes no difficulty and is even simpler here.

The other details are easy to fill and we will not make a line by line copy of 
\cite[Sections 4b)-e)]{BZ}.

\appendix

\section{An extended Rokhlin congruence formula}\label{sA}

In this Appendix, we prove a new Rokhlin type congruence not included in Zhang \cite{Z1}. As mentioned in  Introduction, the mod 2 indices studied in the main text appear most naturally in this version of congruences.

Let $K$ be an $8k+4$ dimensional compact oriented manifold such that $w_2(TK)\neq 0$. Let $B$ be a compact connected codimension two submanifold of $K$ such that $[B]\in H_{8k+2}(K,{\bf Z}_2)$ is the Poincar\'e dual of $w_2(TK)\in H^2(K,{\bf Z}_2)$. We assume the existence of such a submanifold and consider the case where $B$ is non-orientable. 

We fix a spin structure on $K\setminus B$. Then $B$ carries a canonically induced pin$^-$ structure (cf. Kirby-Taylor \cite{KT}).

\begin{rem}\label{ta1}
The case where $B$ is orientable has already been considered in Zhang \cite{Z1}. 
\end{rem}

Let $E$ be a real vector  bundle over $K$. Let $E_{\bf C}$ be the complexification of $E$.

Let $N$ be the normal bundle to $B$ in $K$. Let $e\in H^2(B,o(N))=H^2(B,o(TB))$ be the Euler class of $N$.

Denote by $i:B\hookrightarrow K$ the embedding of $B$ in $K$.

Let ${\rm ind}^t(i^*E)$ be the mod 2 topological index of the real vector bundle $i^*E$ over $B$.

The main result of this Appendix can be stated as follows.

\begin{thm}\label{ta2}
The following identity holds,
\begin{align}\label{a1}
\left\langle\widehat A(TK){\rm ch}\left(E_{\bf C}\right),[K]\right\rangle\equiv {\rm ind}^t\left(i^*E\right)
-\frac{1}{2}\left\langle \widehat A(TB)\tanh\left(\frac{e}{4}\right){\rm ch}\left(i^*E_{\bf C}\right),[B]\right\rangle\ \ ({\rm mod}\ 2).
\end{align}
\end{thm}

\begin{rem}\label{ta3}
Since $\tanh(\frac{e}{4})$ is an odd function in $e$, the characteristic number in the right hand side of (\ref{a1}) is well defined.
\end{rem}

The proof of Theorem \ref{ta2} is almost the same as in Zhang \cite{Z1} with minor modifications. So we only give a sketch.

\begin{proof}
Let $g^{TB}$ be a metric on $TB$. Let $g^N$ be a metric on $N$. Let $\pi:N\rightarrow B$ be the projection map of the normal bundle.

Set $N_1=\{n\in N|\, \|n\|_{g^N}\leq 1\}$, $M=\partial N_1$. Then $N_1$ is a disc bundle over $B$ with fiber $D$, $M$ is a circle bundle over $B$ with fiber $S^1$. The metric $g^N$ restricts to each fiber $S^1$ a metric $g^{S^1}$.

One constructs easily a metric $g^{TD}$ on $TD$ and a series of metrics $g^{TK,\epsilon}$, $\epsilon>0$ on $TK$ such that i). $g^{TK,\epsilon}$ is of product structure near $M$; ii). $g^{TK,\epsilon}|_M=\frac{1}{\epsilon}\pi^*g^{TB}\oplus g^{TS^1}$ and iii). $g^{TK,\epsilon}|_{N_1}=\frac{1}{\epsilon}\pi^*g^{TB}\oplus g^{TD}$. Note $R^{K,\epsilon}$ the curvature of the Levi-Civita connection of $g^{TK,\epsilon}$. 

Let $g^{i^*E}$ be a metric on $i^*E$. Let $\nabla^{i^*E}$ be a Euclidean connection on $i^*E$. $g^{i^*E}$ and $\nabla^{i^*E}$ extend to $i^*E_{\bf C}$ accordingly.

We can then construct a metric $g^E$ and a Euclidean connection $\nabla^E$ on $E$ such that i). $g^E$ and $\nabla^E$ are of product structure near $M$; ii). $g^E|_{N_1}=\pi^*g^{i^*E}$ and iii). $\nabla^E|_{N_1}=\pi^*\nabla^{i^*E}$. 

Let $D_{M,\pi^*i^*E,\epsilon}$ be the Dirac operator associated to $(M,g^{TK,\epsilon}|_M,\pi^*i^*E)$. Then as in \cite[Lemma 3.3]{Z1}, one can apply the Atiyah-Patodi-Singer index theorem for manifolds with boundary \cite{APS} to $K\setminus N_1$ to get the following formula,
\begin{align}\label{a3}
\left\langle \widehat A(TK){\rm ch}\left(E_{\bf C}\right),[K]\right\rangle\equiv \overline \eta\left(D_{M,\pi^*i^*E,\epsilon}\right)
+\left(\frac{1}{2\pi}\right)^{\frac{\dim K}{2}}\int_{N_1}\widehat A\left(R^{K,\epsilon}\right){\rm ch}\left(E_{\bf C},\nabla^{E_{\bf C}}\right)\ \ ({\rm mod}\ 2).
\end{align}

Now since the coupled connection on $\pi^*i^*E$ does not depend on $\epsilon$, the formula of \cite[(3.6)]{Z1} will take the following form in our situation here,
\begin{align}\label{a4}
\lim_{\epsilon\rightarrow 0}\overline\eta\left(D_{M,\pi^*i^*E,\epsilon}\right)\equiv \overline\eta\left(\widetilde D_{B,i^*E}\right)
-\left\langle \widehat A(TB){\rm ch}\left(i^*E_{\bf C}\right)\frac{\tanh\left(\frac{e}{2}\right)-\frac{e}{2}}{e\tanh\left(\frac{e}{2}\right)},[B]\right\rangle\ \ ({\rm mod}\ 2).
\end{align}
And the analogue of \cite[(3.7)]{Z1} turns out to be
\begin{align}\label{a5}
\lim_{\epsilon\rightarrow 0}\left(\frac{1}{2\pi}\right)^{\frac{\dim N_1}{2}}\int_{N_1}
\widehat A\left(R^{K,\epsilon}\right){\rm ch}\left(E_{\bf C}\right)
=\left\langle\widehat A(TB)\frac{1}{e} \left( \frac{\frac{e}{2}}{\sinh\left(\frac{e}{2}\right)}-1\right){\rm ch}\left(i^*E_{\bf C}\right),[B]\right\rangle.
\end{align} 

Formula (\ref{a1}) now follows from (\ref{a3})-(\ref{a5}) and our index theorem Theorem \ref{t4.3}.

The proof of Theorem \ref{ta2} is completed.
\end{proof}

\begin{rem}\label{ta4}
While in general the mod 2 topological index is difficult to calculate, Theorem \ref{ta2} and \cite[Theorem 3.2]{Z1} show that they can be computed through characteristic numbers in some cases.
\end{rem}

\begin{rem}\label{ta5}
Compare with Zhang \cite{Z3}, where we treated the case of orientable $B$. We are not satisfied that in \cite{Z3}, a formula corresponding to (\ref{a1}) here implies the formula corresponding to \cite[Theorem 3.2]{Z1}, while here for the case where $B$ is non-orientable, we don't see such an implication, at least at this moment.
\end{rem}

\begin{rem}\label{ta6}
As we now have formulated the Rokhlin type congruence in a purely topological way, we would like to see a proof of (\ref{a1}) similar to what we have done in \cite{Z3} for the case where $B$ is orientable.
\end{rem}

Now take $E=TK$ in (\ref{a1}). One gets
\begin{align}\label{a6}
\left\langle\widehat A(TK){\rm ch}\left(T_{\bf C}K\right),[K]\right\rangle\equiv {\rm ind}^t\left(TB\oplus N\right)
-\frac{1}{2}\left\langle \widehat A(TB)\tanh\left(\frac{e}{4}\right){\rm ch}\left(T_{\bf C}B\oplus N_{\bf C}\right),[B]\right\rangle\ \ ({\rm mod}\ 2).
\end{align}

On the other hand, by \cite[Theorem 3.2]{Z1} and the index theorem Theorem \ref{t4.3}, one has
\begin{multline}\label{a7}
\left\langle\widehat A(TK){\rm ch}\left(T_{\bf C}K\right),[K]\right\rangle\equiv {\rm ind}^t\left(TB\oplus {\bf R}\oplus o(TB)\right)
-\frac{1}{2}\left\langle \widehat A(TB)\tanh\left(\frac{e}{4}\right){\rm ch}\left(T_{\bf C}B\right),[B]\right\rangle
\\
+\left\langle \widehat A(TB)\frac{ {\rm ch}\left(N_{\bf C}\right)-2\cosh\left(\frac{e}{2}\right)}{2\sinh\left(\frac{e}{2}\right)},[B]\right\rangle\ \ ({\rm mod}\ 2).
\end{multline}

The following  result is a direct consequence of (\ref{a6}) and (\ref{a7}).

\begin{thm}\label{ta7}
The following identity holds,
\begin{align}\label{a8}
{\rm ind}^t(N)-{\rm ind}^t({\bf R}\oplus o(TB))=\left\langle\widehat A(TB)\sinh(e),[B]\right\rangle\ \ ({\rm mod}\ 2).
\end{align}
\end{thm}

\begin{cor}\label{ta8}
For any two dimensional vector bundle $N$ over an $8k+2$ dimensional compact pin$^-$ manifold $B$ verifying that $w_1(N)=w_1(TB)$, the following identity holds,
\begin{align}\label{a9}
2\, {\rm ind}^t(N)\equiv 2\,{\rm ind}^t({\bf R}\oplus o(TB))\ \ ({\rm mod}\ {\bf Z}).
\end{align}
\end{cor}

\begin{proof}
Let $\widetilde B$ be the orientation cover of $B$. Then $N$ lifts to an orientable vector  bundle $\widetilde N$ over $\widetilde B$ with Euler class $\widetilde e$. Now $\widetilde B$ is a spin manifold, so by the classical theorem of Atiyah and Hirzebruch \cite{AH}, the number $\langle\widehat A(T\widetilde B)\sinh(\widetilde e),[\widetilde B]\rangle=2\langle \widehat A(TB)\sinh(e),[B]\rangle$ is an integer. Formula (\ref{a9}) follows from this fact and (\ref{a8}).
\end{proof}

\begin{rem}\label{ta9}
A formula of form (\ref{a8}) for the case where $B$ is spin has been proved in \cite[Corollary 13]{Z3} before. This later result has been extended by Tian-Jun Li \cite{L} to the case where $N$ can be any complex vector bundle over an $8k+2$ dimensional compact spin manifold. It would be interesting to formulate and prove an analogous generalization for (\ref{a8}) here. 
\end{rem}

$\ $

\noindent{\bf Acknowledgements.}   {This work was
partially supported by NSF grant n$^\circ$ DMS 9022140 through an MSRI postdoctoral fellowship.
}

\def\cprime{$'$} \def\cprime{$'$}
\providecommand{\bysame}{\leavevmode\hbox to3em{\hrulefill}\thinspace}
\providecommand{\MR}{\relax\ifhmode\unskip\space\fi MR }
\providecommand{\MRhref}[2]{%
  \href{http://www.ams.org/mathscinet-getitem?mr=#1}{#2}
}
\providecommand{\href}[2]{#2}

\end{document}